\documentclass[11pt]{amsart}
\usepackage{amssymb}
\usepackage{amsmath}
\usepackage{amsthm}
\usepackage{times}
\usepackage{graphicx}

\usepackage[curve,all]{xy}
 \xyoption{curve}
 \xyoption{line}
\xyoption{arc}

\usepackage{color}
\usepackage{amsthm}
\newtheorem{theorem}{Theorem}[section]
\newtheorem{lemma}[theorem]{Lemma}

\newtheorem{proposition}[theorem]{Proposition}

\theoremstyle{definition}

\theoremstyle{remark}
\newtheorem{remark}[theorem]{Remark}
\newtheorem{conjecture}[theorem]{Conjecture}

\numberwithin{equation}{section}

\newcommand{\calC}{\mathcal{C}}
\newcommand{\calD}{\mathcal{D}}

\newcommand{\calH}{\mathcal{H}}

\newcommand{\calM}{\mathcal{M}}
\newcommand{\calN}{\mathcal{N}}
\newcommand{\calO}{\mathcal{O}}

\newcommand{\calX}{\mathcal{X}}

\newcommand{\bbC}{\mathbb{C}}

\newcommand{\bbP}{\mathbb{P}}
\newcommand{\bbE}{\mathbb{E}}
\newcommand{\bbQ}{\mathbb{Q}}
\newcommand{\bbR}{\mathbb{R}}

\newcommand{\bbZ}{\mathbb{Z}}

\newcommand{\la}{\langle}
\newcommand{\ra}{\rangle}

\newcommand{\GL}{\textup{GL}}

\def\Aut{{\text{Aut}}}

\def\Pic{{\text{Pic}}}
\def\Ker{{\text{Ker}}}

\def\PGL{{\text{PGL}}}

\def\cusp{{\text{cusp}}}
\def\En{{\text{En}}}
\def\Co{{\text{Co}}}
\def\deg{{\text{deg}}}
\def\ndeg{{\text{ndeg}}}
\def\nod{{\text{nod}}}

\def\da{\dasharrow}
\def\frako{\mathfrak{o}}
\def\frakC{\mathfrak{C}}
\def\a{{\text{a}}}
\def\Num{{\text{Num}}}

\begin{document}
\title [Coble surfaces and nodal Enriques surfaces] {The rationality of the moduli spaces of Coble surfaces and of nodal Enriques surfaces}

\author{Igor Dolgachev}
\address{Department of Mathematics, University of Michigan, Ann Arbor, MI
48109,USA}
\email{idolga@umich.edu}
\thanks{}

\author{Shigeyuki Kond\=o}
\address{Graduate School of Mathematics, Nagoya University, Nagoya,
464-8602, Japan}
\email{kondo@math.nagoya-u.ac.jp}
\thanks{Research of the second author is partially supported by
Grant-in-Aid for Scientific Research (S), No 22224001, (S), No 19104001}

\begin{abstract}
We prove the rationality of the coarse moduli spaces of Coble surfaces and of nodal Enriques surfaces over the field of complex numbers.
\end{abstract}
\maketitle

\section{Introduction}

The purpose of this note is to prove the rationality of the moduli spaces of Coble surfaces  and of nodal Enriques surfaces over the field of complex numbers. A Coble surface is a rational surface obtained by blowing up 10 nodes of a rational plane curve of degree 6, and an Enriques surface is called nodal if it contains a smooth rational curve. The moduli space of nodal Enriques surfaces is a codimension one subvariety in the 10-dimensional moduli space of Enriques surfaces. When the K3-cover of an Enriques surface degenerates admitting an ordinary double point fixed under an involution, the quotient by the involution is a rational surface obtained from a Coble surface by blowing down  the proper transform of the plane sextic. In this way, the moduli space of Coble surfaces can be identified with a codimension one component of the boundary of the moduli space of Enriques surfaces.

The idea of the proof is similar to the one used by the second author for the proof of rationality of the moduli space of 
Enriques surfaces \cite{Kon}. The K3 surface birationally isomorphic to the double cover of the projective plane branched along the union of a cuspidal plane quintic and its cuspidal tangent contains the lattice $D_8\oplus U$ in its Picard group. It is shown that the moduli space of K3 covers of Enriques surfaces and the moduli space of K3 surfaces admitting this lattice in its Picard group are birationally isomorphic quotients of a bounded symmetric domain of type IV. A similar idea is used here. We prove that the moduli space of K3 
covers of nodal Enriques surfaces (resp. the K3-covers  of Coble surfaces) is birationally isomorphic to the moduli space of K3 surfaces birationally isomorphic to the double cover of the projective plane branched along the  union of a cuspidal plane quintic and its cuspidal tangent line where the quintic 
has an additional double point (resp. the cuspidal tangent line  
 touches the curve  at a nonsingular point). It is easily proved, by using \cite{Mi}, \cite{V}, that the corresponding moduli spaces of plane quintics are rational varieties.

Note that general Enriques and Coble surface are examples of the quotients of a $K3$ surface by a non-symplectic involution which acts identically on the Picard group.  
Recently Ma \cite{Ma} showed the rationality of the moduli spaces of such $K3$ surfaces in many cases.
The case of the  $K3$-covers  of Coble surfaces is one of eight exceptional cases where his methods did not work.

The idea that the moduli spaces of Enriques surfaces (resp. nodal Enriques surfaces, Coble surfaces) should  be related to the moduli space of cuspidal quintics (resp. their special codimension one subvarieties) originates from  some (still hypothetical) purely geometric constructions of the first author which may relate the corresponding moduli spaces. We discuss these constructions in the last two sections of the paper.
\smallskip


\section{Preliminaries}\label{sec2}

A {\it lattice}  is  a free abelian group $L$ of  finite rank equipped with 
a non-degenerate symmetric integral bilinear form $L \times L \to \bbZ$ whose value on a pair $(x,y)$ will be denoted by $x\cdot y$. 
 For $x \in L\otimes \bbQ$, we call $x^2 =x\cdot x$ the {\it norm} of $x$.
For a lattice $L$ and a rational number
$m$, we denote by $L(m)$ the free $\bbZ$-module $L$ with the $\bbQ$-valued bilinear form obtained from the bilinear form of $L$ by multiplication by $m$. The signature of a lattice is the signature of the real vector space $L\otimes \bbR$ equipped with the symmetric bilinear form extended from one on $L$ by linearity. A lattice is called {\it even} if 
$x\cdot x\in 2\bbZ$ 
for all $x\in L$. 
  
We denote by $U$ the even unimodular lattice of signature $(1,1)$, 
and by $A_m, \ D_n$ or $\ E_k$ the even {\it negative} definite lattice defined by
the Cartan matrix of type $A_m, \ D_n$ or $\ E_k$ respectively.  
For an integer $m$, we denote by $\langle m\rangle$ the lattice of rank 1 generated by a vector with norm $m$.  
We denote by $L\oplus M$ the orthogonal direct sum of lattices $L$ and $M$, and by $L^{\oplus m}$ the orthogonal direct sum of $m$-copies of $L$.
For any integer $k$ we denote by $M_k$ the set of $x\in M$ with norm $k$.

We denote by $L_{K3}$ the lattice $E_8^{\oplus 2}\oplus U^{\oplus 3}$. It is isomorphic to the 2-cohomology group $H^2(X,\bbZ)$ of a K3 surface equipped with the structure of a lattice defined by the cup-product. We will refer to $L_{K3}$ as the {\it K3-lattice}.  The lattice $\bbE = E_8\oplus U$ is called the {\it Enriques lattice}. It is isomorphic to the lattice $\Num(S) = \Pic(S)/(K_S)$ of numerical equivalence divisor classes on an Enriques surface $S$.

Let $L$ be an even lattice and let $L^* ={\rm Hom}(L,\bbZ)$ identified with a subgroup  of  $L\otimes \bbQ$ with the extended symmetric bilinear form.  We denote by $A_L$ the quotient
$L^*/L$ and define maps
$$q_L : A_L \to \bbQ/2\bbZ, \quad b_L : A_L \times A_L \to \bbQ/\bbZ$$
by $q_L(x+L) = x\cdot x\ {\rm mod}\ 2\bbZ$ and
$b_L(x+L, y+L) = x\cdot y \ {\rm mod} \ \bbZ$.  
We call $q_L$ the {\it discriminant quadratic form} of $L$ and $b_L$ the {\it discriminant bilinear form}.  
A lattice is called {\it 2-elementary} if its discriminant group is a 2-elementary abelian group. 

Let ${\rm O}(L)$ be the orthogonal group of $L$, that is, the group of isomorphisms of $L$ preserving the bilinear form.
Similarly ${\rm O}(A_L)$ denotes the group of isomorphisms of $A_L$ preserving $q_L$.
There is a natural map
\begin{equation}\label{nat}
\phi:{\rm O}(L) \to {\rm O}(A_L)
\end{equation}
whose kernel is denoted by ${\rm O}(L)^*$.

\section{The moduli spaces of Enriques, nodal Enriques and Coble surfaces}\label{sec3}

First we recall the moduli space of lattice polarized K3 surfaces.
For any even lattice $M$ of signature $(1,r-1)$ primitively embeddable into the K3-lattice $L_{K3}$, one can construct the coarse moduli space 
$\calM_{K3,M}$ (resp.$\calM_{K3,M}^a$)  of isomorphism classes of lattice $M$ polarized (resp. amply polarized) K3 surfaces $X$, i.e. isomorphism classes of pairs $(X,j)$, where $j:M\hookrightarrow {\rm Pic}(X)$ is a primitive lattice embedding such that the image contains a nef and big (resp. ample) divisor class (see \cite{D2}).\footnote{There is some additional technical requirement for the embedding which we refer to loc.cit.} 
Let $N = M_{L_{K3}}^\perp$.  Then the period domain is given by
\begin{equation}\label{period1}
{\calD}(N) = \{ [\omega] \in \bbP(N\otimes \bbC) \ : \  \omega\cdot \omega = 0, \  \omega\cdot \bar{\omega} > 0\}
\end{equation}
which is a disjoint union of two copies of the $20-r$-dimensional bounded symmetric domain of type IV.  
The moduli space is constructed as a quotient
$$\calM_{K3,M} = \calD(N)/{\rm O}(N)^*, \quad  \calM_{K3,M}^{\a} = \bigl(\calD(N)\setminus \calH_{-2}\bigr)/{\rm O}(N)^*,$$
where ${\rm O}(N)^* = \Ker({\rm O}(N)\to {\rm O}(A_{N}))$, and, for any $d\in \bbZ$,
\begin{equation}\label{heegner}
\calH_{-2d} = \bigcup_{\delta\in N_{-2d}}\{[\omega]\in \calD(N) \ : \ \omega\cdot \delta = 0 \}.
\end{equation}
We call $\calH_{-2d}$ the $(-2d)$-{\it Heegner divisor}.
Suppose that $N$ and $M$ satisfy the condition
\begin{equation}\label{surj}
\text{The natural maps  ${\rm O}(N)\to {\rm O}(A_{N}),\ {\rm O}(M)\to {\rm O}(A_{M})$ are surjective}.
\end{equation}
Then
$$\calM_{K3,M}/{\rm O}(A_M) \cong \calD(N)/{\rm O}(N), \quad  \calM_{K3,M}^{\a}/{\rm O}(A_M) = \bigl(\calD(N)\setminus \calH_{-2}\bigr)/{\rm O}(N),$$
are coarse moduli spaces of K3 surfaces wich admit a primitive embedding of $M$ into $\Pic(X)$.

The period point of a marked K3 surface belongs to $\calH_{-2}$ if and only if there exists a primitive embedding of $M\oplus \la -2\ra$ in $\Pic(X)$. The image  of a generator of $\la -2\ra$ will be an effective divisor class $R$ with self-intersection $-2$ such that $R\cdot h = 0$ for every divisor class from the image of $M$. This shows that $X$ does not admit any ample polarization contained in the image of $M$ in $\Pic(X)$. In other words, any nef and ample polarization of $X$ originated from $M$ will blow down $R$ to a double rational point.

Now we consider an Enriques surface $S$.  
Let $\pi:X\to S$ be its K3-cover and let $\sigma$ be the fixed point free involution of $X$.
Then $\pi^*(\Pic(S)) = \pi^*(\Num(S)) \cong \bbE(2)$.   
We take $\bbE(2)$ as $M$ and denote by $N$ the orthogonal complement of $M$ in $L_{K3}$.  Then
\begin{equation}\label{}
N \cong U \oplus \bbE(2).
\end{equation}
Note that $\sigma^*|M = 1_M$ and $\sigma^*N| = -1_N$.  
It is known that any period point $[\omega]$ of the K3-cover $X$ is not contained in $\calH_{-2}$ (e.g. \cite{Na}).
The quotient $\calD(N)/{\rm O}(N)$ is a normal quasi-projective variety of dimension 10, and
$\bigl(\calD(N)\setminus \calH_{-2}\bigr)/{\rm O}(N)$
is the moduli space $\calM_{\En}$ of Enriques surfaces.

Next we consider
{\it nodal} Enriques surfaces, i.e. Enriques surfaces containing a smooth rational curve ($(-2)$-curve, for short). 
Let $C$ be a $(-2)$-curve on an Enriques surface $S$. Then $\pi^{-1}(C)$ splits into the disjoint sum $C_1\cup C_2$ of two $(-2)$-curves. The divisor class $\delta = [C_1-C_2]$ with $\delta^2 = -4$ belongs to $\pi^*(\Pic(S)^\perp)$.   If we consider all $(-2)$-curves on $S$, the corresponding $(-4)$-vectors $\delta$
generate a negative definite lattice $R(2)$ in $U \oplus \bbE(2)$ where $R$ is a root lattice.  The root lattice $R$ is a part of the notion of {\it root invariant}
for Enriques surfaces (see \cite{N3}).
Since any period point $[\omega]$ of the K3-cover $X$ is orthogonal to an algebraic cycle, we obtain that the period $[\omega]$ belongs to $ \calH_{-4}$.
Thus we define the moduli space $\calM_{\En}^{\nod}$ of nodal Enriques surfaces by
\begin{equation}\label{nodmod}
\calM_{\En}^{\nod} = \bigl(\calH_{-4}\setminus \calH_{-2}\bigr)/{\rm O}(U\oplus \bbE(2)),
\end{equation}
where $\calH_{-4}, \calH_{-2}$ are Heegner divisors in the period domain of Enriques surfaces.
It is known that such $(-4)$-vector $\delta$ in $U\oplus \bbE(2)$ is unique up to the orthogonal group ${\rm O}(U\oplus \bbE(2))$, and the orthogonal complement $\delta^{\perp}$
in $U\oplus \bbE(2)$ is isomorphic to 
\begin{equation}\label{latnod}
N = U\oplus \la 4 \ra \oplus E_8(2)
\end{equation}
(see \cite{Na}).  
The orthogonal complement $M = N^{\perp}$ of $N$ in $L_{K3}$ contains $\bbE(2) \oplus \la -4\ra$ as a sublattice of index 2, 
where $\bbE(2) = \pi^*({\rm Pic}(S))$ and $\la -4\ra$ is generated by $\delta$.
Then the quotient $\calD(N)/{\rm O}(N)$ is a 9-dimensional quasi-projective variety.
We call a nodal Enriques surface $S$ is {\it general} if  $R \cong \la -2\ra$, that is, for any two $(-2)$-curves $C, C'$ on $S$, 
$[C_1-C_2] = [C'_1-C'_2]$.
Note that, for a general nodal Enriques surface, the decomposition $\bbE(2) \oplus \la -4\ra$ is unique, that is, it is independent on a choice of $(-2)$-curves.  Hence we have the following.
\begin{proposition}\label{ms2}
Let $N = U\oplus \la 4\ra \oplus E_8(2)$.  Then
the moduli space $\calM_{\En}^{\nod}$ of nodal Enriques surfaces is birationally isomorphic to $\calD(N)/{\rm O}(N)$.
\end{proposition}

Finally we consider Coble surfaces.
A {\it Coble surface} is a smooth rational projective surface $S$ such that $|-K_S| = \emptyset$ but $|-2K_S|\ne \emptyset$ (see \cite{DZ}). A classical example of such a surface is the blow-up of the projective plane at the ten nodes of an irreducible plane curve $C$ of degree 6. The sets of 10 points in the plane realized as the nodes of a rational sextics are examples of \emph{special sets} of points in sense of A. Coble \cite{C2} (they were called Cremona special in \cite{Cantat}). They  were first 
studied by A. Coble in \cite{C1}. In this note we will restrict ourselves only with these classical examples. 

Denote by $\calM_{\Co}$ the moduli space of Coble surfaces constructed as a locally closed subvariety of the GIT-quotient of the variety of 10-tuples of points 
in $\bbP^2$ modulo the group $\PGL(3)$. 
By taking the double cover of $\bbP^2$ branched along the plane sextic with 10 nodes, the moduli space  $\calM_{\Co}$ can be described 
as an open set of an arithmetic quotient of a 9-dimensional bounded symmetric domain of type IV.
We briefly recall this. 

Denote by $X$  the  double cover of the Coble surface $S$ branched along the proper transform of the plane sextic $C$.  Then $X$ is a $K3$ surface containing
the divisors $E_0$, $E_1,\ldots, E_{10}$, where $E_0$ is the pullback of a line on $\bbP^2$ and $E_1,\ldots, E_{10}$ are the
inverse images of the exceptional curves over the nodes $p_1,\ldots, p_{10}$ of $C$.  It is easily seen that the corresponding divisor classes $e_0, e_1,\ldots, e_{10}$ generate
the sublattice $M_X$ of ${\rm Pic}(X)$ isomorphic to $M =\la 2\ra \oplus \la -2\ra^{\oplus 10}$.
Note that $M$ is a 2-elementary lattice of signature $(1,10)$ with $A_M \cong (\bbZ/2\bbZ)^{11}$.
The orthogonal complement of $M_X$ in $H^2(X, \bbZ)$, denoted by $N_X$, is a 2-elementary lattice of signature $(2,9)$ with $q_{N_X} = -q_M$
(see \cite{N1}, Corollary 1.6.2).  The isomorphism class of such lattice is uniquely determined by $-q_M$.  Thus 
$N_X$ is  isomorphic to 
\begin{equation}\label{latcob}
N = \la 2 \ra \oplus \bbE(2)
\end{equation}
(see loc.cit., Theorem 3.6.2).  We remark that $M \cong \bbE(2) \oplus \la -2\ra$.

Let $\calD(N)$ be as in \eqref{period1}, where $N$ is the lattice \eqref{latcob}. 
The quotient $\calD(N)/{\rm O}(N)$ is a normal quasi-projective variety of dimension 9.  
The Torelli type theorem for algebraic $K3$ surfaces, due to Piatetskii-Shapiro and Shafarevich \cite{PS}, implies the following
(for more details, see \cite{MS}):

\begin{proposition}\label{ms1}
Let $N = \la 2 \ra \oplus \bbE(2)$.  
Then the moduli space $\calM_{\Co}$ of Coble surfaces is isomorphic to an open subset of $\calD(N)/{\rm O}(N)$.
\end{proposition}

Note that $N = \la 2\ra \oplus \bbE(2)$ is isomorphic to the orthogonal complement of a $(-2)$-vector in $U \oplus \bbE(2)$.
This implies that the quotient of the $(-2)$-Heegner divisor $\calH_{-2}$ in the period domain of Enriques surfaces by the arithmetic 
subgroup ${\rm O}(U\oplus \bbE(2))$ is birationally isomorphic to the moduli space $\calM_{\Co}$ of Coble surfaces.

\section{Plane quintics with a cusp}\label{sec5}

Let $C$ be a plane quintic curve with a cusp $p$.  Let $L$ be the tangent line of $C$ at the cusp.  We consider the following two cases.

\begin{itemize}
\item [\textbf{Case 1}:] The line $L$ is tangent to $C$ at a smooth point $q$ of $C$. 
\item [\textbf{Case 2}:]  $C$ has an ordinary node $q$. 
\end{itemize}

Let $\calM_{\cusp}$ be the moduli space of cuspidal quintics, that is, the GIT-quotient of the projective space of plane cuspidal curves of 
degree 5 by the group $\PGL(3)$.  The second author proved earlier that $\calM_{\cusp}$ is a rational variety birationally isomorphic to the moduli space $\calM_{\En}$ (see \cite{Kon}). The proof establishes a birational isomorphism between 
$\calM_{\En}$ and the moduli space of K3 surfaces birationally isomorphic to the double covers of $\bbP^2$ branched along a cuspidal quintic. Here we will follow the same strategy replacing $\calM_{\cusp}$ with its codimension 1 subvarieties $\calM_{\cusp}'$ (resp. $\calM_{\cusp}''$) corresponding to quintics from Case 1 (resp. Case 2).

\begin{theorem}\label{Miyata} $\calM_{5,\cusp}'$ and $\calM_{5,\cusp}''$ are  rational varieties of dimension $9$.
\end{theorem}

\begin{proof}  We start with Case 1.

Let $C$ be a quintic curve from Case 1. By a linear transformation,  we may choose coordinates $(x_0:x_1:x_2)$ to assume that  $p = (1:0:0)$ is the cusp, $V(t_1)$ is the cuspidal tangent line which touches $C$ at the point $q = (0:0:1)$. Since $p$ is a cusp of $C$ with cuspidal tangent $V(x_1)$, the curve  $C$ is given by an equation of the form
$$ax_0^3x_1^2 + x_0^2A_1(x_1, x_2) + x_0A_2(x_1,x_2) + A_3(x_1,x_2) = 0, \quad a\ne 0$$
where $A_1,A_2,A_3$ are  homogeneous polynomials of degrees 3, 4, and 5, respectively. Plugging in $x_1= 0$, we obtain the binary form 
$x_0^2A_1(0,x_2)+x_0A_2(0,x_2)+A_3(0,x_2)$ in variables $x_0,x_2$. It  must have a zero  at $(0:1)$ of multiplicity 2. This implies that $A_2 = x_1A_2'$ and $A_3 = x_1A_3'$ for some polynomials $A_2', A_3'$ of degrees 3 and 4, respectively. Thus the equation of $C$ can be rewritten in the form
$$F = ax_0^3x_1^2 + x_0^2A_1(x_1, x_2) + x_0x_1A_2'(x_1,x_2) + x_1A_4'(x_1,x_2) = 0, \quad a\ne 0.$$

Let $V$ be the linear subspace of $S^5({\bf C}^3)^*$ consisting of  quintic ternary forms $F$ as above (with $a$ may be equal to zero). The subgroup $G$ of $\GL(4)$ which leaves invariant $V$ consists of  linear transformations
$$x_0\to ax_0 + bx_2, \quad x_1 \to cx_1,  \quad x_2\to dx_1 + ex_2.$$
Then ${\calM'}_{5,\cusp}$ is birational to the quotient  $\bbP(V)/G$.  It follows that the dimension of $\calM'_{5,\cusp}$ is equal to 9.
Note that $G$ is a solvable algebraic group of dimension 5 acting linearly on the linear space $V$ of dimension 14.  The assertion of the rationality now follows follows from a result of Miyata \cite{Mi} and Vinberg  \cite{V}.

The  Case 2 can be argued in the same way.   We may assume that the node does not lie on the cuspidal tangent line. First transform $C$ to a curve such that   $p = (1:0:0)$ is a cusp with the cuspidal tangent line $V(x_1)$ and $q = (0:1:0)$ is a node.
Arguing as above, we find that $C$ can be given by an equation
$$F = ax_0^3x_1^2 + x_0^2A_1(x_1, x_2) + x_0x_2A_2(x_1,x_2) + x_2^2A_3(x_1,x_2) = 0, \quad a\ne 0.$$
Let  $V'$ be a linear subspace of  $S^5({\bf C}^3)^*$ consisting of  quintic ternary forms $F$ as above (with $a$ may be equal to zero). Its dimension is equal to 13. The subgroup $G'$ of $\GL(4)$ leaving $V'$ invariant consists of  projective transformations
$$x_0 \mapsto ax_0+bx_2, \quad x_1\mapsto cx_1, \quad x_2\to dx_2.$$
It is a solvable algebraic group of dimension 4 acting linearly on $V$. The variety  $\calM''_{5,\cusp}$ is birational to the quotient variety $\bbP(V')/G'$.  It follows that the dimension of $\calM''_{5,\cusp}$ is 9.
Invoking the same result of Miyata and Vinberg, we obtain that 
$\calM''_{5,\cusp}$ is rational.
\end{proof}

\section{$K3$ surfaces associated with a plane quintic with a cusp}\label{sec6}

In this section, we shall show that $\calM'_{5,\cusp}$ (resp. $\calM''_{5,\cusp}$) 
is isomorphic to an open subset of an arithmetic quotient of a 9-dimensional bounded symmetric domain of type IV.  

First we consider Case 1. 

Let $C$ be a cuspidal quintic as in Case 1 and let $L$ be the cuspidal tangent. Consider the plane sextic curve $C + L$. Let $p$ be the cusp and $q$ be the a smooth tangency point of $L$ with $C$.  Let $\bar{X}$ be the double cover of $\bbP^2$ branched along $C+L$.
Then $\bar{X}$ has a rational double point of type $E_7$ over $p$ locally isomorphic to $V(z^2+y(y^2+x^3))$ and a rational double point of type $A_3$ over $q$ locally isomorphic to $V(z^2+x(x+y^2))$.  Denote by $X$ the minimal resolution of $\bar{X}$ and by $\tau$ the covering transformation.  
Then $X$ is a $K3$ surface containing 11 smooth rational curves $E_1,\ldots, E_{11}$ with intersection graph pictured below.  
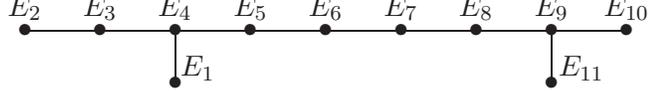
\xy (0,10)*{};(-30,-10)*{};
@={(0,0),(10,0),(20,0),(30,0),(40,0),(50,0),(60,0),(70,0),(80,0), (70,-7),(20,-7)}@@{*{\bullet}};
(0,0)*{};(70,0)*{}**\dir{-};(20,0)*{};(20,-7)*{}**\dir{-};(70,0)*{};(70,-7)*{}**\dir{-};(70,0)*{};(80,0)*{}**\dir{-};
(23,-5)*{E_1};(0,3)*{E_2};(10,3)*{E_3};(20,3)*{E_4};(30,3)*{E_5};(40,3)*{E_6};(50,3)*{E_7};(60,3)*{E_8};(70,3)*{E_9};(80,3)*{E_{10}};(74,-5)*{E_{11}};
\endxy

We see that  
$E_1, \ldots, E_7$ form the intersection graph of type $E_7$,
$E_8$ is the inverse image of $L$ and $E_9, E_{10}, E_{11}$ form the intersection graph of type $A_3$. 
The covering transformation $\tau$ preserves each of $E_1,\ldots, E_9$ and changes $E_{10}$ and $E_{11}$.
Note that the linear system
$$|E_1 + E_3 + 2(E_4 + \cdots + E_9) + E_{10} + E_{11}|$$
defines an elliptic fibration with a singular fiber of type $\tilde{D}_9$, and $E_2$ is a section of this fibration.  
This implies that these 11 curves generate a lattice $M_{X}$ of ${\rm Pic}(X)$ isomorphic to $M = U\oplus D_9$.  Here $U$ is generated by the class of a fiber and the section
$E_2$ and $D_9$ is generated by $E_1, E_4,\ldots, E_{11}$.    
Since the discriminant of $M$ is equal to 4, and  there are no even unimodular lattices with signature $(1,10)$, $M_{X}$ is primitive in $H^2(X, \bbZ)$ and
$M_{X} = {\rm Pic}(X)$ for general $X$.

Let $N_{X}$ be the orthogonal complement of $M_{X}$ in $H^2(X, \bbZ)$.  Then $N_{X}$ has signature $(2, 9)$. 
It follows from  \cite{N1}, Corollary 1.6.2 that $q_{N_{X}} \cong -q_{M}$.  
Note that $A_{M} \cong \bbZ/4\bbZ$.
Also it follows from loc. cit., Theorem 1.14.2 that the isomorphism class of $N_{X}$ is uniquely determined by $q_M$.  
Thus $N_{X}$ is isomorphic to
\begin{equation}\label{lat1}
N = \la 4\ra \oplus U \oplus E_8
\end{equation}
Obviously ${\rm O}(M_{X}) \cong \bbZ/2\bbZ$.
We have the following lemma which is easy to prove.

\begin{lemma}\label{surj1}
The group ${\rm O}(A_{S_{X}})$, and hence ${\rm O}(A_{T_{X}})$, is generated by the covering involution $\tau$.
In particular, the natural maps \eqref{nat}
${\rm O}(M_{X}) \to {\rm O}(A_{M_{X}})$ and ${\rm O}(N_{X}) \to {\rm O}(A_{N_{X}})$
are surjective.
\end{lemma}

Let $\calD(N)$ be as in \eqref{period1} with the lattice $N$ from \eqref{lat1}. The quotient $\calD(N)/{\rm O}(N)$ is a normal quasi-projective variety of dimension 9.  

 We fix a primitive embedding of $M$ into the K3-lattice $L_{K3}$
with $N = M^{\perp}$.  We also fix a basis $\{e_i\}$ of $M$ which has the same incidence relation as  $\{E_i\}$.
It follows from Lemma \ref{surj1} that there exists an isometry from $H^2(X, \bbZ)$ to $L_{K3}$ sending the classes of $E_i$ to $e_i$. 
This defines a  $M$-lattice polarization on the corresponding $K3$ surface (see \cite{D2}) satisfying condition \eqref{surj}.  
Note that the $M$-marking determines the action of the involution of $\tau$ on $H^2(X,{\bf Z})$: $\tau^*$ acts trivially on $e_1,..., e_9$, 
changes $e_{10}$ and $e_{11}$, and acts on $N$ as $-1$.
Conversely let $(X,j), (X',j')$ be two $M$-amply polarized $K3$ surfaces whose periods coincide in $\calD(N)/{\rm O}(N)$.
We denote by $\tau$ (resp. $\tau'$) the involution of $X$ (resp. $X'$). 
It follows from Lemma \ref{surj1} that there exists a Hodge isometry $\phi : H^2(X,{\bf Z}) \to H^2(X', {\bf Z})$ preserving the $M$-markings.
The Torelli type theorem for algebraic $K3$ surfaces implies 
that there exists an isomorphism $\varphi : X\to X'$ with $\varphi^*=\phi$.  Moreover $\varphi \circ \tau = \tau'\circ \varphi$.  Hence $\varphi$
induces an isomorphism between the corresponding plane quintics.  Thus we have 
the following theorem.
\begin{theorem}\label{cusp1}
Let $N = \la 4\ra \oplus U \oplus E_8$.
Then the moduli space ${\calM'}_{5,cusp}$ is isomorphic to an open subset of $\calD(N)/{\rm O}(N)$.
\end{theorem}

Next we study Case 2. 

Consider again the plane sextic curve $C + L$.  
Let $L$ intersect $C$ at the cusp $p$ and two distinct points $q_1, q_2$. 
Let $\bar{X}$ be the double cover of $\bbP^2$ branched along $C+L$. 
Then $\bar{X}$ has a rational double point of type $E_7$ over $p$ and three rational double point of type $A_1$ over $q_1, q_2, q$.  
Denote by $X$ the minimal resolution of $\bar{X}$ and by $\tau$ the covering transformation.  
Then $X$ is a $K3$ surface containing 12 smooth rational curves $E_1,\ldots, E_{12}$ whose intersection graph is pictured below. 

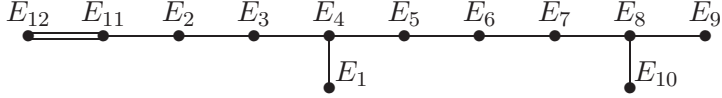
\xy (0,10)*{};(-30,-10)*{};
@={(-10,0),(-20,0),(0,0),(10,0),(20,0),(30,0),(40,0),(50,0),(60,0),(70,0), (60,-7),(20,-7)}@@{*{\bullet}};
(-20,0)*{};(-10,-0)*{}**\dir{=};
(-10,0)*{};(70,0)*{}**\dir{-};(20,0)*{};(20,-7)*{}**\dir{-};(60,0)*{};(60,-7)*{}**\dir{-};(60,0)*{};(70,0)*{}**\dir{-};
(23,-5)*{E_1};(0,3)*{E_2};(10,3)*{E_3};(20,3)*{E_4};(30,3)*{E_5};(40,3)*{E_6};(50,3)*{E_7};(60,3)*{E_8};(70,3)*{E_9};(63.5,-5)*{E_{10}};(-10,3)*{E_{11}};(-20,3)*{E_{12}};
\endxy

 Here $E_9, E_{10}$ or $E_{12}$ corresponds to the exceptional curve over $q_1, q_2$ or $q$, respectively, and 
$E_1,\ldots, E_{7}$ correspond to the exceptional curves over $p$,
$E_8$ is the inverse image of $L$ and $E_{11}$ is the inverse image of the line passing through $p$ and $q$. 
The covering transformation $\tau$ preserves each of $E_1,\ldots, E_{12}$.
Note that the linear system
$$|2(E_4 + \cdots + E_8) + E_{1} + E_{3}+E_{9} + E_{10}|$$
defines an elliptic fibration with a singular fiber of type $\tilde{D}_8$ and of type $\tilde{A}_1$, and $E_{2}$ is a section of this fibration.  
This implies that these 12 curves generate the sublattice $M_{X}$ of $\Pic(X)$  isomorphic to $M = U\oplus D_8\oplus A_1$.  
Here $U$ is generated by the class of fiber and the section
$E_{2}$, the sublattice $D_8$ is generated by $E_1,E_4,\ldots, E_{10}$,  and the sublattice $A_1$ is generated by $E_{12}$. 
Since the fixed locus of $\tau$ consists of a smooth  curve of genus 4 and four  smooth rational curves $E_2, E_4, E_6, E_8$,
the invariant sublattice of $H^2(X, \bbZ)$ under the action of $\tau^*$ coincides with $M_{X}$ (\cite{N2}, Theorem 4.2.2).
In particular, $M_{X}$ is primitive in $H^2(X, \bbZ)$ and
$M_{X} = {\rm Pic}(X)$ for generic $X$.

Let $N_{X}$ be the orthogonal complement of $M_{X}$ in $H^2(X, \bbZ)$.  Then $M_{X}$ has signature $(2, 9)$. 
$N_{X}$ is isomorphic to
\begin{equation}\label{lat2}
N = \la 2\ra \oplus U(2) \oplus E_8,
\end{equation}
It is easy to see that ${\rm O}(A_{M_{X}}) \cong \bbZ/2\bbZ$.
This implies the following Lemma.

\begin{lemma}\label{surj2}
The group ${\rm O}(A_{M_{X}})$ is generated by the isometry of $S_{X}$ acting trivially on $E_3,\ldots, E_{12}$ and switching $E_1$ and $E_2$.
The group ${\rm O}(A_{N_{X}})$ is generated by the covering involution $\tau$.
In particular, the natural maps
${\rm O}(M_{X}) \to {\rm O}(A_{M_{X}})$ and ${\rm O}(N_{X}) \to {\rm O}(A_{N_{X}})$
are surjective.
\end{lemma}

Let $\calD(N)$ be as in \eqref{period1} with $N$ defined in \eqref{lat2}. 
Using the same argument as in Case 1, we prove the following theorem.

\begin{theorem}\label{cusp2}
Let $N = \la 2\ra \oplus U(2) \oplus E_8$.
Then the moduli space $\calM''_{5,\cusp}$ is isomorphic to an open subset of $\calD(N)/{\rm O}(N)$.
\end{theorem}

\begin{remark}\label{}
Ma \cite{Ma} proved the rationality of the moduli space of $K3$ surfaces in the case 2 by using  another model of the quotient $X/\tau$.
\end{remark}

\section{Proof of the rationality}\label{sec7}

In this section, we prove the rationality of $\calM_{\Co}$ and $\calM_{\En}^{\nod}$. 

\begin{theorem}\label{main1} Let $N$ (resp. $N'$) be the lattice from \eqref{latcob} (resp. \eqref{lat1}). 
Then
$${\calD}(N)/{\rm O}(N)\cong {\calD}(N')/{\rm O}(N').$$
\end{theorem}

\begin{proof}
We have $N(1/2) = \la 1\ra \oplus U \oplus E_8$.  The odd unimodular lattice $N(1/2)$ contains the unique even sublattice of index 2 isomorphic
to $N'$.  Thus we can consider $N'$ as a sublattice of $N(1/2)$.  Then any isometry of $N(1/2)$ preserves $N'$ and hence 
${\rm O}(N(1/2)) \subset {\rm O}(N')$.  Conversely, consider 
$N'{}^* = \la \frac{1}{4}\ra \oplus U \oplus E_8$.  The discriminant group $A_{N'}$ is a finite cyclic group of order 4 and contains the unique subgroup
$N(1/2)/N'$ of order 2.  This implies that any isometry of $N'$
can be extended to the one of $N(1/2)$ and hence ${\rm O}(N') \subset {\rm O}(N(1/2))$.  Thus we have ${\rm O}(N(1/2)) = {\rm O}(N')$.
Now consider the bounded symmetric domain $\calD(N(1/2)) = \calD(N')$.  Then  
$$\calD(N(1/2))/{\rm O}(N(1/2)) = \calD(N')/{\rm O}(N').$$
Obviously 
$$\calD(N(1/2))/{\rm O}(N(1/2)) \cong \calD(N)/{\rm O}(N).$$  
Therefore we have proved the assertion.
\end{proof}

\begin{theorem}\label{main2} Let $N$ (resp. $N'$) be the lattice from \eqref{latnod} (resp. \eqref{lat2}). Then 
$${\calD}(N)/{\rm O}(N) \cong {\calD}(N')/{\rm O}(N').$$
\end{theorem}

\begin{proof}
Consider 
$$N'(1/2) = \la 1\ra \oplus U \oplus E_8(1/2), \quad (N'(1/2))^* = \la 1 \ra \oplus U \oplus E_8(2).$$
Then $N$ is the even sublattice of $(N'(1/2))^*$.  Hence ${\rm O}(N'(1/2)) \subset {\rm O}(N)$.
Conversely, consider $N^* = \la \frac{1}{4}\ra \oplus U \oplus E_8(1/2)$ and the discriminant quadratic form
$$q_{N} : N^*/N \to \bbQ/2\bbZ.$$
We remark that $N'(1/2)$ is characterized as being the maximal submodule $K$ of $N^*$ such that $q_{N}(K/N) \subset \bbZ/2\bbZ$. 
Any isometry of $N$ can be extended to the one, denoted by $\phi$, of $N^*$.  Then the above remark implies that $\phi$ preserves $N'(1/2)$.
Hence ${\rm O}(N) \subset {\rm O}(N'(1/2)))$.
Therefore 
$$\calD(N)/{\rm O}(N) = \calD(N'(1/2))/{\rm O}(N'(1/2)) \cong \calD(N')/{\rm O}(N').$$  
\end{proof}

\noindent
Combining Propositions \ref{ms1},  \ref{ms2} and Theorems \ref{cusp1}, \ref{cusp2}, \ref{main1}, \ref{main2}, we have the following theorem.

\begin{theorem}\label{main3} There are birational isomorphisms
\begin{eqnarray}\label{bir}
\Upsilon&:&\calM_{\En}\cong \calM_{5,\cusp},\\ 
\Upsilon'&:& \calM_{\Co}\cong \calM'_{5,\cusp},\\ 
\Upsilon''&:& \calM_{\En}^{\nod}\cong \calM''_{5,\cusp}.
\end{eqnarray}
\end{theorem}

By Theorem \ref{Miyata}, we obtain the following main theorem.

\begin{theorem}\label{main4}
$\calM_{\Co}$ and $\calM_{\En}^{\nod}$ are rational varieties.
\end{theorem}

\begin{remark} The K3-cover $X$ of a general nodal Enriques surface is isomorphic to a minimal nonsingular model  of \emph{Cayley quartic symmetroid},  the locus $Y$ of singular quadrics in a general 3-dimensional linear system $L$ of quadrics in $\bbP^3$ (see \cite{Cos}). The surface $Y$ has 10 nodes corresponding to reducible quadrics in $L$. The set of 10 points in $\bbP^3$ realized as the twn nodes of a Cayley quartic symmetroid is one of special sets of points in $\bbP^3$ in the sense of A. Coble (see \cite{C2}). There is a beautiful relationship between Cayley quartic symmetroids and rational sextics (see loc.cit.). The variety of such sets modulo projective equivalence is birationally isomorphic to the GIT-quotient of the Grassmannian $G_3(L)$ of webs of quadrics in $\bbP^3$ modulo $\PGL(4)$. It is birationally isomorphic to some finite cover of $\calM_{\En}^{\nod}$. The rationality of $G_3(L)/\PGL(4)$ is a difficult problem.
\end{remark}

\section{A geometric construction: Enriques surfaces} Let $\calM_{\En}(2)$ be the moduli space of degree 2 
polarized Enriques surfaces, i.e. the coarse moduli space of pairs $(X,h)$, where $h$ is a nef divisor class  with $h^2 = 2$.  It is known that $h = F_1+F_2$, where $F_1,F_2$ are nef divisors with $F_i^2 = 0$ and $F_1\cdot F_2 = 1$, or $h = 2F_1+R$, where $F_1$ is as above and $R$ is $(-2)$-curve with $F_1\cdot R = 0$ (see \cite{CD}, Corollary 4.5.1). We call 
$h$ \emph{non-degenerate} if $h$ is as in the first case, and \emph{degenerate} otherwise. 

 For any $h$ as above, the linear system $|2h|$ defines a degree 2 map $\phi_h:X\to \bbP^4$ whose image is a quartic del Pezzo surface. If $h$ is non-degenerate, $\mathsf{D}$ has 4 ordinary double points, otherwise, it has 2 ordinary double points, and one rational double point of type $A_3$. We call $\mathsf{D}$ a 4-nodal quartic del Pezzo surface in the first case and a degenerate 4-nodal quartic del Pezzo surface in the second case 
(see \cite{CD}, Chapter 0, \S 4).  The set of fixed points of the deck transformation $\sigma$ of the double cover $\phi_h$  consists of a smooth  curve $\overline{W}$  and 4 isolated points. The image $W$ of $\overline{W}$ on $\mathsf{D}$ is a curve of arithmetic genus 5 from the linear system $|\calO_{\mathsf{D}}(2)|$. It does not pass through the singular points of $\mathsf{D}$. The map $(\phi_h)_{|W}:\overline{W}\to W$ is the normalization map.

 Let $\calM_{\En}(2)^{\ndeg}$ be the GIT-quotient  $|\calO_{\mathsf{D}}(2)|/\!/\Aut(\mathsf{D})$, where $\mathsf{D}$ is a non-degenerate quartic del Pezzo surface, and 
let $\calM_{\En}(2)^{\deg}$ be the same when $\mathsf{D}$ is a degenerate quartic del Pezzo surface. 
The first variety (resp. second one) is a projective variety of dimension 10 (resp. 9). The disjoint union 
$$\calM_{\En}(2) = \calM_{\En}(2)^{\ndeg}\cup \calM_{\En}(2)^{\deg}$$
can be viewed  as a compactification of the moduli space of degree 2 polarized Enriques surfaces. Consider the forgetting rational maps
$$\Phi_1:\calM_{\En}(2)^{\ndeg} \da \calM_{\En}, \quad \Phi_2:\calM_{\En}(2)^{\deg}\da \calM_{\En}^{\nod}.$$
It is known that $\Phi_1$ is of degree $2^{7}\cdot 17\cdot 31$ (see \cite{BP}) and $\Phi_2$ is of degree $2^3\cdot 17$ (see \cite{CD2}).

 Any quartic del Pezzo surface is equal to the base locus of a pencil of quadrics in $\bbP^4$. If $\mathsf{D}$ is a 4-nodal quartic del Pezzo surface, the pencil contains three singular quadrics:  two quadrics of corank 2 and one quadric of corank 1.  If $\mathsf{D}$ is a degenerate 4-nodal quartic del Pezzo surface, it contains only two singular quadrics, both of corank 2 (see \cite{CD}, Chapter 0, \S 4). The locus $C$ of singular quadrics in the net $\calN = |I_{W}(2)|$ of quadrics 
containing $W$ is a curve of degree 5. It has two singular double points corresponding to the singular quadrics of corank 2 
containing $\mathsf{D}$. The line $\ell$ joining the two singular points is the pencil $|I_{\mathsf{D}}(2)|$ of quadrics containing $\mathsf{D}$. In the case of a non-degenerate polarization, the line $\ell$ intersects $C$ at some other nonsingular point $q$. In the case of a degenerate polarization, $\ell$ is 
tangent to a branch of one of the singular points. Let $\bar{C}$ be the normalization of $C$. We assume that $W$ is nonsingular; this happens for any unnodal Enriques surface and for a general nodal surface. In this case  the curve $\bar{C}$ is a nonsingular curve of genus 4. Its plane quintic model is given by the linear system $|K_{\bar{C}}-\bar{q}|$, where $\bar{q}$ is the pre-image of $q$ on the normalization. Recall that the canonical model of a nonsingular curve of genus 4 is the complete intersection of a cubic surface and a quadric surface. The rulings of the quadric define two 
$g_3^1$'s on the curve (they coincide if the quadric is a cone). In our case, the plane quintic model has two singular points, the pencils of lines through these points define two different $g_3^1$'s. Thus the quadric containing the canonical model of $\bar{C}$ is nonsingular. Each of the two lines passing through the point $\bar{q}$ intersect $\bar{C}$ at two points corresponding to the branches of the singular points of $C$. Thus, in the degenerate case, the point $\bar{q}$ coincides with one of the ramification points of the corresponding $g_3^1$. For a general canonical curve of genus 4 we have 12 ramification points in each $g_3^1$.

Let $\calX_4$ be the coarse moduli space of pairs $(T,t)$ which consist of a nonsingular curve $T$ of genus 4 and a 
point $t$ on it. Let $\calX_4'$ be a hypersurface in $\calX_4$ of pairs $(T,t)$ such that $t+2t'$  belongs to a $g_3^1$ on $T$. The forgetting map $(T,t) \mapsto t$ defines a map $\calX_4\to \calM_4$, where $\calM_4$ is the coarse 
moduli space of curves of genus 4. In the non-degenerate case, it identifies (birationally) $\calX_4$ with the universal curve over $\calM_4$. In the degenerate case, it is a finite cover of degree 24.  Projecting $T$ from the point $t$, we obtain a plane quintic curve with a node and a cusp. Conversely, the normalization of such curve defines a pair $(T,t)$ as above, 
where $t'$ corresponds to the branch of the cusp and $t$ is the pre-image of the residual point of the line joining the two 
singular points. In this way we obtain a birational isomorphism
$$\calX_4' \cong \calM_{5,\cusp}''.$$

By assigning to $(X,h)\in \calM_{\En}(2)^{\ndeg}$ (resp. $(X,h)\in \calM_{\En}(2)^{\deg}$) the quintic plane curve $C$ parameterizing singular quadrics of the branch curve of the degree 2 map $\phi_h:X\to \mathsf{D}$,  we obtain a rational map
$$\Psi_1:\calM_{\En}(2)^{\ndeg}\da \calX_4,$$

$$({\rm resp.} \ \Psi_2:\calM_{\En}(2)^{\deg}\da \calX_4'.)\footnote{The varieties $\calX_4,\calX_4'$ are known to be rational \cite{Catanese},
the variety $\calM_{\En}(2)^{\ndeg}$ is also known to be rational.}$$
The degree of the map is equal to the number of projective equivalence classes of nets of quadrics with the fixed curve of singular quadrics. It is known that the number of such equivalence classes  is equal to the number of non-effective even theta characteristics on the normalization of the 
curve (\cite{Beau}, Chapter 6). Since the canonical model of $\bar{C}$ lies on a nonsingular quadric, all even characteristics are non-effective, and their number is equal to $2^3(2^4+1) = 2^3\cdot 17$. Since the degrees of the maps $\Psi_2$ and $\Phi_2$  coincide, it is natural to make the following.

\begin{conjecture} Let $\Upsilon''$ be the birational isomorphism in Theorem \ref{main3}. 
Then 
 $$\Psi_2 = \Upsilon''\circ \Phi_2.$$
\end{conjecture}

Next we assume that $h$ is a non-degenerate degree 2 polarization. Let $(T,t)\in \calX_4\setminus \calX_4'$. 
Assume that $\bar{q}$ 
is not  a ramification point of a $g_3^1$ on $T$. Consider the rational map 
$$f:\bbP^3\da \bbP^4$$
 given by the linear system of cubic surfaces containing $T$. We choose a basis of the linear system in the form
$(V(F),V(x_0Q),\ldots,V(x_3Q))$, where $T = V(F)\cap V(Q)$ is the intersection of a cubic and a quadric and $x_0,\ldots,x_3$ are projective coordinates in $\bbP^3$.  The image of the map is a 
singular cubic hypersurface $K$ with equation
$$y_0Q(y_1,\ldots,y_4)+F(y_1,\ldots,y_4) = 0.$$
 The singular point $\frako = [1,0,\ldots,0]$ of $K$ is the image of the quadric $V(Q)$.  Let $L$ be the tangent to 
$T$ at the point $t$. Its image on $K$ is a line $\ell$ not containing the point $\frako$. Consider the projection of $K$ to the plane 
from the line  
$\ell$. Let  $K'\to K$ be the blowing up of $K$ along $\ell$.  The projection defines a structure of a conic bundle on $K'$. It is known that the discriminant curve is of degree 5 (see \cite{Beau}). We claim that in our case it is a cuspidal quintic. Of course, the cusp is the 
projection of  the point  $\frako$.
It is clear that the cusp distinguishes from an ordinary node by the property that the pencil of lines through the point contains only one line intersecting the curve at the point with multiplicity $\ge 3$. A line through the singular point of the discriminant curve 
corresponds to a hyperplane $H$ containing the line $\ell$ and the singular point $\frako$. Let $y = f(t)$ be the image of the point $t$ on $K$. Its 
coordinates are $(0:a_0:\ldots:a_3)$, where $t = (a_0:\ldots:a_3)$.  Let $\ell' = \overline{\frako y}$ be the line on $K$ joining 
$y$ with $\frako$.\footnote{In \cite{D3}, the first author made a mistake by choosing the line
 $\ell'$ on $K$ instead of $\ell$ for the projection map.} It is easy to see that the plane in $H$ spanned by $\ell$ and $\ell'$ 
intersects $K$ along the line $\ell'$ taken with multiplicity 2. This plane corresponds to the projection of 
$\frako$, i.e. the cusp  of the discriminant curve. The cubic surface $S = H\cap K$ is the image of a plane $\Pi$  in $\bbP^3$ which contains the tangent line $L$. Assume that 
$L$ is not tangent to the quadric $V(Q)$. The  restriction of the map $f$ to $\Pi$ is given by the net  of cubic 
curves which are tangent to each other at the point $t$, the base points of the net lie on the conic $V(Q)\cap \Pi$. 
Since $\Pi$  contains the tangent line to $V(Q)$, the conic $V(Q)\cap \Pi$ is tangent to $L$ at $t$, and hence equals the union of two 
lines $l_1,l_2$ intersecting at $t$. The cubic $V(F)$ intersects the conic at $t$ and for additional points $p_1,p_2\in l_1$ and $p_3,p_4\in l_2$. The conic bundle on $S$ contains four singular conics: $2\ell'$, and the  images of the 
reducible conics  $\overline{p_1p_2}+\overline{p_3p_4}, \overline{p_1p_3}+\overline{p_1p_4}, \overline{p_1p_4}+\overline{p_2p_3}$. The surface S has 2 singular points of type $A_1$, one is the node of $K$, and another lies on $\ell'$.  We see that any plane $\Pi$ not tangent to $V(Q)$ has four singular conics.  So, there is only one line intersecting  the discriminant curve at less than 4 points. It corresponds to the plane $\Pi$ tangent to the quadric $V(Q)$ at the point $t$. This proves the assertion.

Let $\calC$ be the coarse moduli space of pairs $(K,\ell)$, where $K$ is a cubic threefold with at one ordinary double point and $\ell$ is a line not containing the singular point. It is known that the map $\calC\to \calM_{5,\cusp}$ which assigns to $(K,\ell)$ the discriminenant curve of the conic bundle defined by $\ell$ is of degree  equal to the number of odd theta characteristics on the normalization of the 
discriminant curve \cite{Beau}, Remark 6.27. The latter number is equal to $2^4\cdot (2^5-1) = 2^4\cdot 31$. Thus our construction defines a rational map of 
degree $2^4\cdot 31$
$$\Psi_2:\calX_4\da \calM_{5,\cusp}.$$
Composing it with the rational map $\Psi_1:\calM_{\En}(2)^{\ndeg}\da \calX_4$ we get a rational map
$$\Psi_2\circ \Psi_1:\calM_{\En}(2)^{\ndeg} \da \calM_{5,\cusp}$$
of degree $2^7\cdot 17\cdot 31$. Comparing this with degree of the map 
$$\Phi_1:\calM_{\En}(2)^{\ndeg}\da \calM_{\En},$$
we propose the following (see \cite{D3}).

\begin{conjecture} Let $\Upsilon$ be the birational isomorphism from \eqref{bir}. Then
$$\Upsilon\circ \Phi_1 = \Psi_2\circ \Psi_1.$$
\end{conjecture}

\section{A geometric construction: Coble surfaces} 

We assume that $S$ is an unnodal Coble surface in the sense that it has no $(-2)$-curves. We know that the orthogonal complement of 
$K_S$ in $\Pic(S)$ isomorphic to the Enriques lattice $\bbE$. As in the case of unnodal Enriques surfaces, we consider a 
polarization $h = [F_1+F_2]$ of degree 2. Such polarizations correspond  to  lattice embeddings $U\hookrightarrow \bbE$. One can show 
that the linear system $|2h|$ defines a regular map $\phi_h:S\to \bbP^4$ whose image is a 4-nodal quartic del Pezzo surface $\mathsf{D}$ (see 
\cite{Cantat}). The set of fixed points of the deck transformation $\sigma$ of $S$  consists of a smooth curve $\overline{W}$ of genus 4 and three 
isolated points. The image of the anticanonical curve $\frakC\in |-2K_S|$ is a singular point $q$ of $\mathsf{D}$ (other three singular points are the images of the isolated fixed points). The curve $\overline{W}$ intersects $\frakC$ at two points. Its image $W$ on $\mathsf{D}$ is a 
curve of arithmetic genus 5 with a double point at $q$. It is equal to the complete intersection of a net of quadrics. Thus a Coble surface is obtained as a degeneration of an Enriques surface when the branch curve $W$ of the map $\phi_h$ passes through a singular point of 
$\mathsf{D}$.

Let us look at the discriminant curve of the net $\calN$ of quadrics with base locus equal to $W$. It is a plane quintic with two double points corresponding to corank 2 quadrics in the pencil $|\calO_{\mathsf{D}}(2)|$. What is different here is that one of the double points is a cusp. In fact, it is 
known that the tangent cone of a double point of the discriminant curve corresponding to a corank 2 quadric in the net consists of 
quadric tangent to the singular line $\ell$  of the corank 2 quadric. Since the point $q$ is a base point of the net $\calN$, the restriction of $\calN$ to the singular line $\ell$ has $q$ as a base point. Hence there is only one pencil of quadrics touching $\ell$ at one point (different from $q$). This proves the assertion. Let $T$ be the normalization of the discriminant curve $D$ of $\calN$. As in the case of Enriques surfaces, it is 
isomorphic to the intersection of a nonsingular quadric $V(Q)$ and a cubic surface $V(F)$. The quintic curve $D$ is obtained by 
projection of $T$ from a point $t$ on $T$. Since the projection has a cusp, the point $t$ is a residual point of a ramification point of one of the $g_3^1$ on $T$. So, as in the case of nodal Enriques surfaces we obtain a rational map
$$\Psi_1': \calM_{\Co}(2)\da \calX_4'.$$
The degree of this map is the same as the degree of the map $\Psi_1$, i.e. equal to $2^3\cdot 17$. 

Next we consider the restriction of the map $\Psi_2:\calX_4\da \calM_{5,\cusp}$ to $\calX_4'$. In the notation of the previous section, 
the cubic $V(F)$ intersects the tangent line $L$ at the point $t$ and is tangent at the point $p_2=p_3 = t'$. The 
cubic surface $S$ is the image of the plane $\Pi$ under the rational map given by the net of cubics tangent at the point $t$,  tangent to 
the line $\overline{tt'}$ at $t'$ and passing through the points $p_4,p_5$. So, we have only two singular conics on $S$ in the pencil of conics defined by the hyperplane $H$. This shows that the line defined by $H$ intersects the discriminant curve at two point, with multiplicity 3 at the cusp and multiplicity 2 at some other point. This gives a rational map
$$\Psi_2':\calX_4'\da \calM_{5,\cusp}'.$$
The degree of this map is the same as in the Enriques case, i.e. equal to $2^4\cdot 31$. The composition map
$$\Psi_2'\circ \Psi_1':\calM_{\Co}(2) \da \calM_{5,\cusp}'$$
is of degree $2^7\cdot 17\cdot 31$. On the other hand, since the automorphism group of a general Coble surface is isomorphic to the automorphism group of a general Enriques surface (see \cite{Cantat}), the same count as in \cite{BP} shows that the degree of the rational map
$$\Phi':\calM_{\Co}(2) \da \calM_{\Co}$$
is equal to $2^7\cdot 17\cdot 31$. This suggests the following.

\begin{conjecture} Let $\Upsilon'$ be the birational isomorphism in Theorem \ref{main3}.
Then 
$$\Upsilon'\circ \Phi' = \Psi_2'\circ \Psi_1'.$$
\end{conjecture}

  \bibliographystyle{amsplain}

\end{document}